\documentclass[12pt]{amsproc}
\newtheorem{theorem}{\sc Theorem}[section]
\newtheorem{lemma}[theorem]{\sc Lemma}
\newtheorem{proposition}[theorem]{\sc Proposition}

\usepackage{color}

\begin{document}
\title[Coprime and anti-coprime  commutators]
{Strong conciseness of coprime and anti-coprime commutators}

\author{Eloisa Detomi}

\address{Dipartimento di Ingegneria dell'Informazione - DEI, Universit\`a di Padova, Via G. Gradenigo 6/B, 35121 Padova, Italy} 
\email{eloisa.detomi@unipd.it}

\author{Marta Morigi}
\address{Dipartimento di Matematica, Universit\`a di Bologna, 
Piazza di Porta San Donato 5, 40126 Bologna, Italy}
\email{marta.morigi@unibo.it}

\author{Pavel Shumyatsky}
\address{Department of Mathematics, University of Brasilia,
Brasilia-DF,   70910-900 Brazil }
\email{pavel2040@gmail.com}

\keywords{Commutators, conciseness, profinite groups, pronilpotent groups}
\subjclass{20E18; 20D25.}

\begin{abstract}
A coprime  commutator in a profinite group $G$ is an  element of the form $[x,y]$, where $x$ and $y$ have coprime order and  an anti-coprime commutator is  a  commutator $[x,y]$ such that the orders of $x$ and $y$ are divisible by the same primes.  In the present paper we establish  that a profinite group $G$ is finite-by-pronilpotent if  the cardinality of the set of coprime commutators in $G$ is 
  less than $2^{\aleph_0}$. 
Moreover, 
 a profinite group $G$ has finite commutator subgroup $G'$  if the cardinality of the set of anti-coprime commutators in $G$ is 
  less than $2^{\aleph_0}$. 
\end{abstract} 

\maketitle
\section{Introduction}
Let $w=w(x_1, \dots, x_k)$ be a group-word, that is, a nontrivial element of the free group on $x_1, \dots, x_k$. The word $w$ can be naturally viewed as a function  of $k$ variables defined on any group $G$. The subgroup generated by the set of $w$-values is called the verbal subgroup corresponding to the word $w$. We denote this by $w(G)$. In the context of topological groups $G$ we write $w(G)$ to denote the closed subgroup generated by all  $w$-values in $G$. 

The word $w$ is said to be concise in the class of groups  $\mathcal{C}$ if  the verbal subgroup  $w(G)$  is finite whenever $w$ takes only finitely many values in a group $G \in \mathcal{C}$ (see for example \cite{DMSconc1}, \cite{DMSconc2}  and references therein for results on conciseness of words). A natural variation of this notion for profinite groups was recently introduced in \cite{ben-pavel}: we say that $w$ is strongly concise in a class of profinite groups $  \mathcal{C}$ if  the verbal subgroup  $w(G)$ is finite in any group $G \in \mathcal{C}$  in which $w$ takes less than   $2^{\aleph_0}$ values. A number of new results on strong conciseness of group-words can be found in \cite{ben-pavel}. 

The concept of (strong) conciseness can be applied in a much  wider context. Suppose $  \mathcal{C}$ is a class of profinite groups and $\varphi(G)$
 is a subset of $G$ for every  $G \in  \mathcal{C}.$ One can ask whether the subgroup generated by $\varphi(G)$ is finite whenever  $|\varphi(G)| < 2^{\aleph_0}.$
 
 In the present paper we establish strong conciseness of the sets of coprime and anti-coprime commutators. By a coprime commutator in a profinite group $G$ we mean any element of the form $[x,y]$, where $\pi(x) \cap \pi(y) = \emptyset,$ and by an anti-coprime commutator we mean any commutator $[x,y]$ such that $\pi(x) = \pi(y)$.   Here and throughout the paper $\pi(x)$ stands for the set of prime divisors of the order of the procyclic group generated by  $x $. 

It is well-known that in any profinite group $G$ the coprime commutators generate the pronilpotent residual $\gamma_\infty (G)$, that is, the (unique) minimal normal subgroup $N$ such that $G/N$ is pronilpotent.
So, if a profinite group is finite-by-pronilpotent, then it has only finitely many coprime commutators. Our main result is that also the converse holds. 

\begin{theorem}\label{main}  
A profinite group $G$ is finite-by-pronilpotent if and only if the cardinality of the set of coprime commutators in $G$ is 
  less than $2^{\aleph_0}$. 
\end{theorem}


Remark that in the particular case of the group $G$ with only finitely many coprime commutators the above theorem is straightforward from the main result of 
\cite{AST}. 

The set of anti-coprime commutators generates  the commutator subgroup $G'$ of $G$. Hence if $G'$ is finite, then so is this set. 
As a consequence  of the results in  \cite{ben-pavel}, we have the following new characterization of finite-by-abelian groups. 
 
 \begin{theorem}\label{anti-main}  
A profinite group $G$ has finite commutator subgroup $G'$  if and only if the cardinality of the set of anti-coprime commutators in $G$ is 
  less than $2^{\aleph_0}$. 
\end{theorem}
 
 It follows from the proof of Theorem \ref{anti-main} that if the profinite group $G$ has only finitely many, say $m$, anti-coprime commutators, then the order of $G'$ is $m$-bounded. Throughout the paper we use the expression ``$m$-bounded" to mean that a quantity is finite and bounded by a certain number depending only on  the parameter $m$.

To conclude this short introduction, we mention  a related problem. Let  $\gamma_k(G)$  denote the  $k$th term of the lower central series of a group $G$ and $G^{(k)}$  that of the derived series of $G$. In view of our results it would be interesting to find a small strongly concise set of generators for each of these subgroups.

\section{Preliminaries}\label{prel}  

Many  results of the theory of finite groups admit a natural interpretation for
profinite groups. Throughout the paper we use certain profinite versions of
facts on finite groups without explaining in detail how the results on profinite
groups can be deduced from the corresponding finite cases. On all such
occasions the deduction can be performed via the routine inverse limit argument.

Our notation 
 is  standard (see, for example, \cite{rib-zal} or \cite{wilson}). Every homomorphism of profinite groups considered in this paper is continuous, and every subgroup of a profinite group is closed.

If $G$ is a profinite group, then $|G|$ denotes its order, which is a  Steinitz number, and $\pi(G)$ denotes the set of prime divisors of $|G|$. Similarly, if $g$ is an element of $G$, then $|g|$ and $\pi(g)$ respectively  denote the order of the procyclic subgroup generated by $g$ and the set of prime divisors of $|g|$.


Profinite groups have $p$-Sylow subgroups and satisfy analogues of the Sylow theorems.
Prosoluble groups have  $\pi$-Hall  subgroups and Sylow bases. 
 If $G$ is a profinite group and $\pi$ a set of primes,  $O_{\pi} (G)$ stands for  the maximal normal pro-$\pi$ subgroup of $G$.

 Each profinite group $G$ has a maximal pronilpotent normal subgroup, its Fitting subgroup $F(G)$. Set $F_0(G)=1$ and $F_{i+1}(G)/F_i(G)=F(G/F_i(G))$ for every integer $i\ge 0$. 
  A profinite group $G$ 
 is metapronilpotent if and only if it $F_2(G)=1$ 
   or, equivalently, if and only if $\gamma_\infty (G)$ is pronilpotent.

 One can easily see that if $N$ is a normal subgroup of $G$ and $g$ is an element whose image in $G/N$
is a coprime  (resp. anti-coprime) commutator,
then there exists a coprime  (resp. anti-coprime) commutator $h$ such that $g\in hN$. 
  It follows  that if a  profinite group has   less than $2^{\aleph_0}$ coprime  (resp. anti-coprime) commutators, then this condition is
inherited by  every homomorphic image of the group, and we shall use this property without special reference. 
\begin{lemma}\cite[Lemma 2.2]{ben-pavel}  \label{lem:fin-conj-cl} Let $G$ be a profinite group and let $x \in G$.  If the conjugacy class $x^G$ contains less than $2^{\aleph_0}$ elements, then it is finite.
\end{lemma}

The following lemma is an important observation used throughout this paper. 

\begin{lemma} \label{lem:oss0} Let $G$ be a profinite group  in which the set   of coprime commutators  contains less than $2^{\aleph_0}$ elements. 
 If $x \in G$ and $H$ is a subgroup of $G$ such that $(|x|, |H|)=1$,   then $C_H(x)$ has finite index in $H$. 
\end{lemma}
\begin{proof} 
Since $[x,h]=x^{-1}x^h$,  the cardinality of the set $x^H=\{x^h \mid h \in H \}$   is equal to the cardinality of the subset $\{ [x,h] \mid h \in H\}$ of 
 the set of coprime commutators of $G$.  Thus $x^H$  contains less than $2^{\aleph_0}$ elements. 
As $x^H$ is in bijection with the coset space $H/ C_H(x)$, a homogeneous profinite space, we conclude that $x^H$ is finite. 
\end{proof}

If a group $A$ acts on the group $G$ via automorphisms,  the subgroup generated by elements of the form $g^{-1} g ^a$ with $g\in G$  and $a\in A$ is denoted by $[G, A]$. It is well-known that 
the subgroup $[G, A]$ is an $A$-invariant normal subgroup of $G$. 
The symbol $\Phi(G)$ stands for the Frattini subgroup of a group $G$. 

The following result is well-known and will be often used without reference.

\begin{lemma}\cite[Lemma 4.29]{isaacs}\label{coprime} Let $A$ act via automorphisms on $G$, where $A$ and $G $ are
finite groups such that $(|G|, |A|) = 1$. Then $[G, A, A] = [G, A].$\end{lemma}

\begin{lemma}\label{Ncentral}
 Let $G=QH$ be a  group, where $H$ is a subgroup and 
  $Q$ is a normal subgroup of $G$
  such that $Q=[Q,H]$. If $N$ is a normal subgroup of $Q$ such that $[N,H]=1$, 
   then $N$ is contained in the center of $G$.
\end{lemma}
\begin{proof}
By the Three Subgroup Lemma (see \cite[5.1.10]{robinson}) we have that $[Q,H,N]\le [H,N,Q][N,Q,H]\le[H,N]=1$. Thus, $N$ centralizes $[Q,H]=Q$ and the result follows.
\end{proof}


\begin{lemma} \cite[Lemma 2.3]{ijac16} \label{coprimePhi}
Let $V$be an elementary abelian finite $p$-group, and $U$ a $p'$-group of automorphisms
of V$ $. If $|[V, u]|\le m$ for every $u\in U$ , then $|[V, U ]|$ is $m$-bounded, and therefore $|U |$ is
also $m$-bounded. \end{lemma}

\begin{lemma} \label{zero} Assume that $G=QH$ is a finite group with a normal  nilpotent subgroup $Q$ and a subgroup $H$ 
such that   $(|Q|,|H|)=1$ and  $|Q:C_Q(x)|\leq m$ for all $x\in H$. Then the order of $[Q,H]$ is $m$-bounded.
\end{lemma}
\begin{proof} As $[Q,H,H]=[Q,H]$ by Lemma \ref{coprime}, we can replace $Q$ with $[Q,H]$ and simply assume that $Q=[Q,H]$. Consequently, as $Q$ is the direct product of its Sylow subgroups, $P=[P,H]$ for each Sylow subgroup $P$ of $Q$. As $|P:C_P(x)|\le|Q:C_Q(x)|\leq m$, it follows that $H$ acts trivially on $P$ whenever $p>m$. Taking into account that $P=[P,H]$ we conclude that all prime divisors of the order of $Q$ are at most $m$. Hence, the number of Sylow subgroups of $Q$ is $m$-bounded. Therefore, without loss of generality, we may assume that $Q$ is a $p$-group and $H$ acts faithfully on $Q$.  
 As $(|Q|,|H|)=1$, by  \cite[(24.1)]{asch}  the action on $Q/\Phi(Q)$ is faithful.   
 So, by Lemma \ref{coprimePhi} we obtain that $|H|$ is $m$-bounded.
It follows that $|Q:C_Q(H)|$ is $m$-bounded. 
 Let $N$ be the intersection of all conjugates in $G$
of $C_Q(H)$. Then $N$ is a normal subgroup of $G$ of $m$-bounded index. 
We deduce from Lemma \ref{Ncentral} that $N$ is a subgroup of the center of $G$.
 Thus, Schur's theorem  \cite[4.12]{robinson2} tells us that $G'$ has finite $m$-bounded order, and the lemma follows.
\end{proof}

\section{Coprime commutators}

In this section we analize the structure of profinite groups  with  less than $2^{\aleph_0}$  coprime commutators.
 We wish to prove that  
  these commutators generate a finite subgroup. 

\begin{lemma} \label{ph} Let $G$ be a profinite group  with  less than $2^{\aleph_0}$  coprime commutators. 
Assume that $G$ is a product of a subgroup $H$ and a  normal pronilpotent  subgroup $Q$  with $(|Q|,|H|)=1$. 
 Then $[Q,H]$ is finite.
\end{lemma}
\begin{proof} By the profinite version of Lemma \ref{coprime} we have  $[Q,H,H]=[Q,H]$. So we can replace $Q$ with $[Q,H]$ and assume that $Q=[Q,H]$. By Lemma \ref{lem:oss0},   
 $|Q:C_Q(x)|$ is finite for any $x\in H$. 
For each positive integer $j$ consider the set:
$$C_j=\{x\in H| \;|x^Q|\le j\}.$$
Note that the sets $C_j$ are closed 
(see for instance \cite[Lemma 5]{LP}).
As the union of the sets $C_j$ is  $H$, by the Baire category theorem (cf. \cite[p.200]{Ke}) at least one of the sets $C_j$ has non-empty interior  with respect to $H$.
So there exists a positive integer $s$, an open normal subgroup $J$ of $H$ and an element $a \in H$ 
such that $|(ax)^Q|\le s$ for every $x\in J$. For $x\in J$ and $y\in Q$ write $x^y=(a^{-1})^y(ax)^y$. Note that  $|(a^{-1})^Q|\le s$ and $|(ax)^Q|\le s$.
Therefore $|Q:C_Q(x)|\leq s^2$ for any $x\in J$. By Lemma \ref{zero} we conclude that $[Q,J]$ is finite.

Replacing $Q$ with  the quotient group  $Q/[Q,J]$ and $H$ with the quotient group
$H/J$, we can assume that $H$ is finite and thus also $|Q:C_Q(H)|$ is finite. Note that $C_Q(H)$ has only finitely many conjugates in $G$.
Let $N$ be the intersection of all conjugates of $C_Q(H)$. Clearly, $|Q:N|$ is finite. 
  It follows from  Lemma \ref{Ncentral}  that $N$ is a subgroup of the center of $G$. 
Thus the center of $G$ has finite index in $G$, and the result now follows from the theorem 
of Schur \cite[4.12]{robinson2}.
\end{proof}


\begin{lemma}\cite[Lemma 2.4]{AST}\label{ast} If $G$ is a finite metanilpotent group, 
then $\gamma_\infty(G)=\prod_p[K_p,H_{p'}]$, where $K_p$ is a $p$-Sylow 
 subgroup of $\gamma_\infty(G)$ and $H_{p'}$ is a $p'$-Hall subgroup of $G$.
\end{lemma}

\begin{lemma}\label{metapro} Let $G$ be a metapronilpotent  group  with  less than $2^{\aleph_0}$  coprime commutators. Then  for each $p\in\pi(\gamma_\infty(G))$ the $p$-Sylow subgroup of $\gamma_\infty(G)$ is finite.
\end{lemma}
\begin{proof} Note that $\gamma_\infty(G)$ is pronilpotent. Choose $p\in\pi(\gamma_\infty(G))$ and let $P$ be the $p$-Sylow subgroup of $\gamma_\infty(G)$. The profinite version of Lemma \ref{ast} implies that $P=[P,H]$, where $H$ is a $p'$-Hall subgroup of $G$. Lemma \ref{ph} now shows that $P$ is finite, as required.
\end{proof}

Recall that any prosoluble group $G$ has a Sylow basis (a family of pairwise permutable
$p_i$-Sylow subgroups $P_i$ of $G$, exactly one for each prime $p_i \in \pi(G)$), and any two Sylow bases are conjugate (see \cite[Proposition 2.3.9]{rib-zal}). The basis normalizer (also known as the system normalizer) of such a Sylow basis in $G$ is $T =\cap_{i} N_G(P_i).$ If $G$ is a prosoluble group and $T$ is a basis normalizer in $G,$ then $T$ is pronilpotent and  $G = \gamma_\infty (G) T$ (see \cite[Lemma 5.6]{reid}).

\begin{lemma}\label{tttt}  Let $G$ be a metapronilpotent profinite group  with  less than $2^{\aleph_0}$  coprime commutators.
 Then $\gamma_\infty(G)$ is finite.
\end{lemma}
\begin{proof} Let $K=\gamma_\infty (G)$.  In view of Lemma \ref{metapro} it is sufficient to show that  $\pi(K)$ is finite.

Assume by contradiction that $\pi(K)$ is infinite.
 Let $T$ be a basis normalizer of a Sylow basis in $G$. Note that both $T$ and $K$ are pronilpotent.

Let  $\pi(K)=\{p_1, p_2, \dots, p_i, \dots \}$ 
 and for each $i=1,2, \dots$  let $P_i$ be the $p_i$-Sylow subgroup of $K$. By Lemma \ref{metapro}  we know that the subgroups $P_i$ are finite. 
 For each index $i$ observe that the group $O_{p_i'}(T)$ naturally acts on $P_i$.
Let $A_i$ be the group of automorphisms of $P_i$ induced by the action of $O_{p_i'}(T)$. It follows from Lemma \ref{ast} that  $P_i=[P_i,A_i]$. Set $\sigma_i=\pi(A_i)$. 
Then $\sigma_i$ is a finite nonempty set. Let  $q\in\pi(T)$ and let $T_q$ be a $q$-Sylow  subgroup of $T$. By Lemma \ref{ph},
the subgroup $[O_{q'}(K),T_q]$ is finite. Therefore $q\in\sigma_j$ for only finitely many indices $j$. We denote the set $\{p_i\}\cup\sigma_i$ by $\tau_i$.
So for every index $i$ there exists only finitely many indexes $j$ such that $\tau_j\cap\tau_i\ne\emptyset$.

Since $\pi(K)$ is infinite, we can choose an infinite set of indices $J$ such that $\tau_i\cap\tau_j=\emptyset$ whenever $i,j$ are different 
indices in $J$. For each $i\in J$ choose  $q\in\sigma_i$ and let $Q_i$ be the $q$-Sylow subgroup of $T$.
It follows that $[P_i,Q_i]\neq1$ while $[P_j,Q_i]=1$ whenever $i,j \in J$ are such that $i \not = j$.

For each $i\in J$ choose $a_i\in P_i$ and $b_i\in Q_i$ such that $1\neq[a_i,b_i]=c_i$.
Obviously the element $c_i$ is a coprime commutator. 
Observe that by construction $\langle a_i,b_i\rangle$ commutes with $\langle a_j,b_j\rangle$ 
whenever $i\neq j$. Therefore for any subset $I\subseteq J$ the product $\prod_{i\in I}c_i$ is a coprime commutator. 
Since the set $J$ is infinite, we obtain that the cardinality of the set  of coprime commutators of the form $\prod_{i\in I}c_i$ is $2^{\aleph_0}$. 
This is a contradiction and so the lemma follows.
\end{proof}

We are now ready for the proof of  Theorem \ref{main}. 

\begin{proof}[Proof of Theorem \ref{main}]  
Let $G$ be a profinite group. Clearly, if $G$ is finite-by-pronilpotent, then  the cardinality of the set of coprime commutators in $G$ is finite. 

We need to show that if $G$ is  a profinite group with less than $2^{\aleph_0}$ coprime commutators, then $G$ is finite-by-pronilpotent. We can assume that $G$ is infinite.  

Let us first show that the Fitting subgroup $F=F(G)$ of $G$ is nontrivial. Assume for a contradiction that $F=1$. 
 If $x$ is any nontrivial coprime commutator, then  by Lemma \ref{lem:fin-conj-cl} the conjugacy class $x^G$ is finite. Hence 
$|G: C_G(x)| $ is finite and the same holds for any of the (finitely many) conjugates of $x$ in $G$. Therefore, $|G: C_G( \langle x^G \rangle)|$ is finite and so 
 $\langle x^G\rangle$ is central-by-finite. 
Taking into account that $F=1$ we deduce that $\langle x^G\rangle$ is finite for any coprime commutator $x$. Thus, $G$ possesses minimal 
normal subgroups. 
 Let $N$ be the product of all minimal normal subgroups of $G$. 
If $N$ is finite, then 
 there exists an open  normal subgroup $K$ of $G$ such that $K \cap N=1$. 
 In particular, $K$ has no nontrivial minimal normal subgroups. 
 Note that $K$ is not pronilpotent, since $F=1$. So $K$ contains a nontrivial coprime commutator and, by the above argument, its Fitting subgroup is 
 nontrivial. 
  This  contradicts the assumption that $F=1$ and therefore $N$ is infinite. 

 Thus $N$ is the Cartesian product of infinitely many nonabelian finite simple groups $S_i$, with $i$ ranging
in a set of indices $J$. 
Note that in each $S_i$ there exists a nontrivial coprime commutator $c_i$. Namely, take a nontrivial central element $x_i$ in one Sylow $p$-subgroup of $S_i$, for some prime $p$; then $x_i$ cannot commute with all Sylow $q$-subgroups of $S_i$, for all $q\ne p$, because otherwise $x_i$ would be central in $S_i$. 
 For any subset $I\subseteq J$ the product $\prod_{i\in I}c_i$ is a coprime commutator. 
Since the set $J$ is infinite, we obtain that the cardinality of the set  of coprime commutators of the form $\prod_{i\in I}c_i$ is 
at least $2^{\aleph_0}$. 
This is a contradiction and so we conclude that  $F\neq1$. 

 Observe that $F$ actually must be infinite. Indeed, if 
 $F$ is finite, 
 then there exists a normal subgroup $H$ of finite index such that $H\cap F=1$. 
 On the other hand, by the above argument, the Fitting subgroup of $H$ is 
 nontrivial,  
 a contradiction. 
  Thus $F$ is infinite.

By Lemma \ref{tttt},   $\gamma_\infty(F_2(G))$ is finite. Let $R$ be an open normal subgroup 
of $F_2(G)$ intersecting  $\gamma_\infty(F_2(G))$ trivially.
As $R$ is pronilpotent, it follows that $F$ has finite index in $F_2(G)$. Since $F_2(G)/F=F(G/F)$,
we deduce that $G/F$ has
finite Fitting subgroup. Since we have proved above that an infinite profinite group with less than $2^{\aleph_0}$ 
coprime commutators has infinite Fitting subgroup, we conclude that $G/F$ is finite. 

Now we argue by induction on the index of  $F$  in $G$, the case where $G= F$ being trivial. 
If there exists a proper normal subgroup $N$ of $G$  properly containing $F$, then $|N: F(N)| < |G:F|$.  By induction, $\gamma_\infty(N)$ is finite. 
 We use the bar notation in the quotient group $\bar{G}=G/\gamma_\infty(N)$. As $|\bar G: F(\bar G)|< |G:F|$, 
by induction we deduce that $\gamma_\infty (\bar{G})$ is finite, hence $\gamma_\infty (G)$ is finite as well. 

If $G/F$ is abelian,  then $G$ is metapronilpotent and in view of Lemma \ref{tttt} the claim is proved. 

So we are left with the case where the quotient $G/F$ is nonabelian simple. Let $p$ be a prime divisor of $G/F$. 
Choose a 
$p$-element $g \in G\setminus F$. Since $F \langle g \rangle $ is metapronilpotent, 
by  Lemma \ref{tttt} we conclude that $\gamma_\infty (F \langle g \rangle ) $ is finite. 
 As $\gamma_\infty (F \langle g \rangle ) $ is normal in $F$, the normalizer of $\gamma_\infty (F \langle g \rangle ) $ 
in $G$ has finite index in $G$. Thus $\langle \gamma_\infty (F \langle g \rangle ) \rangle^G $ is 
the product of finitely many conjugates of $\gamma_\infty (F \langle g \rangle ) $, which are all normal in $F$.
 Therefore,  $\langle \gamma_\infty (F \langle g \rangle ) \rangle^G $ is finite. 
 Passing to the quotient over $\langle \gamma_\infty (F \langle g \rangle ) \rangle^G $, we can assume that  $F \langle g \rangle $ is pronilpotent. 
Let $\{t_1, \dots , t_s\}$ be a transversal of $F$ in $G$ and note that 
  each $F \langle g^{t_i}  \rangle $ is pronilpotent.  
 So $W=O_{p'}(F)$  centralizes  each $g^{t_i}$. Moreover $W$  centralizes the $p$-Sylow subgroup $P$ of  $F$.

Let $G_1=\langle g^{t_1}, \dots , g^{t_s} \rangle$ and note that $G=FG_1=WPG_1$. 
As $[W, PG_1]=1,$ it follows that $PG_1$ is normal in $G.$
 
 Note that if $ \gamma_\infty ( PG_1)=1$, then $G$ is pronilpotent, being a product of two normal pronilpotent subgroups. Hence  $\gamma_\infty (G)= \gamma_\infty ( PG_1)$ and, since $F \cap PG_1=F(PG_1)$, we can assume without loss of generality that $G=PG_1$. 
 We still have that $G/F$ is a nonabelian simple group and now $O_{p'}(F)$  is central in $G$. 
 Since $G/F$ is a nonabelian  simple group, we can choose a prime $q \neq p$ and a $q$-element $h$ in $G \setminus F$ and argue as above. Thus, given  a  transversal $\{r_1, \dots , r_k\}$  of $F$ in $G$, we can assume that  each $F \langle h^{r_i}  \rangle $ is pronilpotent. Let $G_2=\langle h^{r_1}, \dots , h^{r_s} \rangle$ and note that $G=FG_2$. Now we get that the Hall $q'$-subgroup $O_{q'}(F)$ of $F$ centralizes  each $h^{r_i}$, hence $[O_{q'}(F),G_2]=1$. Moreover $O_{q'}(F)$  centralizes the $q$-Sylow subgroup $Q$ of  $F$, hence $$[O_{q'}(F), QG_2]=1.$$  In particular $QG_2$ is normal in $G=FG_2=O_{q'}(F)QG_2$ and thus $\gamma_\infty (G)= \gamma_\infty ( QG_2)$. Moreover $F(QG_2)=F\cap QG_2$. 
  Note that  both the $p'$-Hall subgroup and the  $q'$-Hall subgroup of $F(QG_2)$ are central in $QG_2$, hence $F(QG_2)$ is central as well. Since $ F(QG_2)$ has finite index in $QG_2$, it follows from the theorem of Schur that $\gamma_\infty (QG_2)=\gamma_\infty(G)$ is finite, as claimed. 
\end{proof}

\section{Anti-coprime commutators}

In this section we prove that a profinite group with less than $2^{\aleph_0}$  anti-coprime commutators has finite derived subgroup.

In  \cite{ben-pavel} the strong conciseness of several words was established. Among these are the simple commutator $[x,y]$ and the $2$-Engel word $[x,y,y]=[[x,y],y]$. Thus, we have 

\begin{proposition}\label{strong} Let $w$ be either the word $[x,y]$ or the $2$-Engel word $[x,y,y]$ and let $G$ be a profinite group. Then, the verbal subgroup $w(G)$ is finite if and only if $w$ takes less than $2^{\aleph_0}$ values in $G$. 
 \end{proposition}

\begin{proof}[Proof of Theorem \ref{anti-main}] 
As observed in the introduction, it is sufficient to show that the derived subgroup $G'$ of a profinite group $G$ is finite whenever $G$ has 
 less than $2^{\aleph_0}$  anti-coprime commutators. So, assume that $G$ is a group with this property.

Observe that, for all $g,h \in G$, the element $[g,h,h]$ is an anti-coprime commutator, since $$[[g,h],h]=[(h^{-1})^g h, h]=[(h^{-1})^g, h]^h.$$ Therefore, since $G$ has  less than $2^{\aleph_0}$  anti-coprime commutators, there are less than $2^{\aleph_0}$ values of the $2$-Engel word in $G$. It follows from Proposition \ref{strong} that the verbal subgroup $N$ generated by the values of the $2$-Engel word is finite. Passing to the quotient $G/N$, we can assume  that  $[g,h,h]=1$ for every $g,h \in G$, so $G$ is a $2$-Engel group. Since finite Engel groups are nilpotent (see \cite[12.3.4]{robinson}), the group $G$ is pronilpotent. 
  
As $G$ is a direct product of its  Sylow subgroups, every  simple commutator $[g,h]$ is an  anti-coprime  commutator. Indeed, we can always write $h=h_1 h_2$ and $g=g_1g_2$, where  $\pi(g_1)=\pi(h_1)$ while $(|h_2|,|g|)=(|h|,|g_2|)=1$. Then simply observe that $[g,h]=[g_1,h_1]$ is an anti-coprime commutator. 
  
So $G$ has less than $2^{\aleph_0}$ simple commutators $[g,h]$. Using Proposition \ref{strong} 
we  conclude that $G'$ is finite, as required. 
\end{proof}

\begin{center}{\textbf{Acknowledgments}}
\end{center}
The authors thank the referee for his suggestions.
The first and second author are members of INDAM and were  partially supported by BIRD185350/18;  the third author was partially supported by FAPDF and CNPq-Brazil.\\ \\
\vskip 0.4 true cm

\end{document}